\documentclass[11pt]{amsart}
\usepackage{amsmath}
\usepackage{amssymb}
\usepackage{graphicx}
\usepackage{bm}
\usepackage{hyperref}
\newtheorem{theorem}{Theorem}[section]
\newtheorem{corollary}[theorem]{Corollary}
\newtheorem{lemma}[theorem]{Lemma}
\newtheorem{proposition}[theorem]{Proposition}
\theoremstyle{definition}
\newtheorem{definition}[theorem]{Definition}

\newtheorem{remark}[theorem]{Remark}
\newtheorem{claim}{Claim}[subsection]

\numberwithin{equation}{section}
\newtheorem*{acknowledgement*}{Acknowledgments}

\newenvironment{DIof}[1]{\noindent{\sc Proof #1.\,}}{\vspace{.05cm} \nobreak \hfill $\square$ \\}

\newcommand{\cw}{\stackrel{\mathcal{G}}{\rightharpoonup}}
\newcommand{\dx}[1]{\mathrm{d}{#1}}
\newcommand{\eqdef}{\stackrel{\mathrm{def}}{=}}

   \def\R{\mathbb{R}}
   
   \def\N{\mathbb{N}}
   \def\Z{\mathbb{Z}}

   \def\supp{\mathop{\rm supp}\nolimits}

   \def\loc{\mathop{\rm loc}\nolimits}

   \def\wlim{\mathop{w-\rm{lim}}}

   \newcommand{\beq}{\begin{equation}}
   \newcommand{\eeq}{\end{equation}}
   \newcommand{\ba}{\begin{array}}
   \newcommand{\ea}{\end{array}}
   \def\a {{\alpha}}

\def\i {\mathbf{i}}

\newcommand{\pairing}[2]{\langle #1,#2 \rangle}

\title[Magnetic Schr\"odinger]{Nonlinear Schr\"odinger equation with bounded magnetic field}

\author[G. Devillanova]{Giuseppe Devillanova}
\address{Politecnico di Bari, via Amendola, 126/B, 70126 Bari, Italy}
\email{giuseppe.devillanova@poliba.it}
\author[C. Tintarev]{Cyril Tintarev}
\address{Technion - Israel Institute of Technology, Haifa 3200, Israel }
\email{tammouz@gmail.com}

\thanks{One of the authors (G.D.) is supported by GNAMPA of the Istituto Nazionale di Alta Matematica (INdAM)
and by MIUR-FFABR-2017 research grant.
Another author (C.T.) thanks the mathematics department of the Politecnico of Bari for their warm hospitality. He completed this work during his stay at Technion as Lady Davis Visiting Professor. He also acknowledges access to library resources of Uppsala University.}
\date{\today}
\begin{document}

\begin{abstract}
The paper studies existence of solutions for the nonlinear Schrödinger equation
\begin{equation}\label{eq:magschr}
-(\nabla+\i A(x))^2u+V(x)u=f(|u|)u	
\end{equation}
with a general bounded external magnetic field. In particular, no lattice periodicity of the magnetic field or presence of external electric field is required.
Solutions are obtained by means of a general structural statement about bounded sequences in the magnetic Sobolev space. 
\end{abstract}
\maketitle


\section{\label{intro}Introduction}
The present paper studies existence of solutions for the nonlinear Schrödinger equation with bounded external magnetic field $B$ on $\R^N$ without additional assumptions, such as lattice periodicity. External magnetic field enters Schr\"odinger equation, as well as other equations of quantum mechanics, by addition of a real-valued covector field $A$, called the magnetic potential to the momentum operator $\frac{\hbar}{\i}d$. Magnetic field, which is a measurable quantity, is a differential 2-form $B=d A$, while the magnetic potential is defined up to an arbitrary additive term $d\varphi$, where $\varphi:\R^N\to \R$ is an  arbitrary scalar function. Magnetic momentum 
$\frac{\hbar}{\i} d+A$ has the following gauge invariance property:
\begin{equation*}\label{eq:shifts0}
	(\frac{\hbar}{\i} d+A)(e^{\i\frac{\varphi}{\hbar}}u)=
	e^{\i\frac{\varphi}{\hbar}}
	(\frac{\hbar}{\i}d+A+d\varphi)u.
\end{equation*}
In this paper we follow the convention that normalizes the mass of the particle and sets the Planck constant $\hbar$ to be equal to $1$ (which is always possible by rescaling the time and space variables).

Up to our knowledge, the earliest existence result for nonlinear magnetic Schr\"odinger equation 
is the paper by Esteban and Lions \cite{LionsMag} where the magnetic field is assumed to be constant. Their approach was generalized to periodic magnetic field by Arioli and Szulkin \cite{ArioliSzulkin} (see also \cite{SchinTinMag}, \cite{CIS}, \cite{BegoutSchindler}). In addition to that, there is a number of important existence results for quasiclassical solutions, that is, solutions of the nonlinear magnetic Schr\"odinger equation that exist if $\hbar$ is sufficiently small. The earliest study of quasiclassical solutions known to us is due to Kurata  \cite{Kurata}, where existence was obtained under assumptions that involved both electric and magnetic potential, and it was followed by a number of other results concerning quasiclassical solutions by Cingolani, Jeanjean, Secchi, Tanaka and several others, see \cite{CJS, Cingolani16, Cingolani17} and references therein. In most of these works existence of solutions is connected to concentration at critical points of the electric potential as the Planck constant tends to zero. Several other papers studied critical nonlinearities in the scalar field term, \cite{BCS, CS}, and the Aharonov-Bohm field, related to critical nonlinear growth, in \cite{ClappSzulkin}. 
The main technical difficulty in proving existence results for the magnetic Schr\"odinger equation is the lack of compactness of Sobolev embeddings in the whole $\R^N$, and it is overcome by the use of a concentration-compactness argument. For instance, in \cite{ArioliSzulkin}, one controls the loss of compactness in problems with a periodic magnetic field by means of energy-preserving operators
\begin{equation}
\label{eq:m-shifts}
g_y := \; u\mapsto e^{\i\varphi_y(\cdot)}u(\cdot-y), \quad y\in\Z^N,
\end{equation}
(where $\varphi_y$ is a suitable re-phasing function (see \eqref{eq:phiy} below)),
known as \emph{magnetic shifts} with their inverse
\begin{equation*}
\label{eq:m-invshifts}
g^{-1}_y := \; v\mapsto e^{-\i\varphi_y(\cdot+y)}v(\cdot+y), \quad y\in\Z^N\,.
\end{equation*}
Another concentration mechanism applies to quasiclassical asymptotics and is not considered here. The present paper aims to extend the approach of \cite{LionsMag} and \cite{ArioliSzulkin} to problems with a generic bounded magnetic field, without any periodicity assumption. This is achieved by a refined concentration-compactness analysis (see Theorem \ref{thm:MagneticPD} and Theorem \ref{thm:energies}) that uses a suitable non-isometric counterpart of magnetic shifts (see Definition \ref{def:g}). 
This analysis allows to prove a variety of existence results which are beyond the scope of the present paper. Our objective here is to describe the concentration mechanism in the non-periodic case and give a sample existence result which, in particular, includes existence of solutions in absence of an electric field.

The paper is organized as follows. Section~2 is dedicated to notation, as well as to conventions for magnetic shifts in the lattice-periodic case. In Section~3 we provide the generalization of magnetic shifts for the non-periodic case, as well as the notions of the magnetic field and of the corresponding energy functional at infinity. In Section~4 we prove a ``vanishing lemma'', that asserts, in terms of generalized magnetic shifts, a property of the Sobolev embedding similar to cocompactness.
In Section~5 we employ generalized magnetic shifts to express the defect of compactness (i.e.\ the difference between the weak limit (modulo subsequence) of a bounded sequence $(u_k)_{k\in\N}$ and a particular profile decomposition (i.e.\ a sum of terms with asymptotically disjoint supports, each of them having the form of suitable magnetic shifts (determined by mutually diverging sequences of points) applied to {\em{a fixed function}}). In Section~6 we apply the profile decomposition obtained in Section~5 to show the existence of nontrivial solutions to \eqref{eq:magschr} under assumption that the magnetic field is bounded and vanishes at infinity. Finally, in the Appendix, we give some details about the set of magnetic shifts in the lattice periodic case and show that it generates a group modulo re-phasings.

Main results of this paper are Theorem~\ref{thm:MagneticPD} and Theorem~\ref{thm:energies} concerning profile decompositions in magnetic Sobolev space and Theorems~\ref{thm:minimaxT} and \ref{thm:minimaxV} providing some representative existence results. It should be noted, that while most existence results in literature pose conditions on the magnetic potential which is not a measurable quantity and is not uniquely defined by the magnetic field, this paper provides existence of solution at a specified energy level using  simple conditions on the magnetic field itself: a uniform bound and the vanishing at infinity.

\section{\label{notation}Notation and basic properties}
In this paper the set $\N$ of natural numbers is fixed as starting with zero.
Given a magnetic potential $A$ as a real-valued linear form, we define 
\begin{equation}\label{eq:nablaA}
\nabla_A (u)\eqdef  (\nabla+ \i A)u,
\end{equation}
and the scalar product
\begin{equation*}\label{eq:scalar}
\pairing{u}{v}_A\eqdef\int_{\R^N} \nabla_A u(x) \overline{\nabla_A v(x)}\dx{x}\,,
\end{equation*}
(the symbol of the scalar product on $\R^N$ will be neglected throughout the paper). 
Then, we introduce the (Dirichlet-type) energy functional
\begin{equation}\label{eq:hamiltonian}
E_A(u)\eqdef \int_{\R^N}|\nabla_A u(x)|^2\dx x.
\end{equation}

Note that the following relation holds true.
\begin{equation}\label{eq:norm2}
\begin{split}
|\nabla_A u|^2 & =|\nabla u|^2+\i A(u\nabla\bar{u}-\bar{u}\nabla u)+|A|^2|u|^2\\
& \geq |\nabla u|^2-4|A||u||\nabla u|+|A|^2|u|^2\,.
\end{split}
\end{equation}
Indeed, by applying Cauchy-Schwarz and triangle inequality, the real number
\begin{equation}\label{products}
\begin{split}
\i A (u\nabla\bar{u}-\bar{u}\nabla u)  & =
2 \i^2 A(\Re(u)\nabla\Im(u)-\Im(u)\nabla\Re(u))\\
& \geq -2 |A|(|\Re(u)|\,|\nabla\Im(u)|+|\Im(u)|\,|\nabla\Re(u))|)\\
& \geq -4 |A|\,|u|\,|\nabla u|\,.\\
\end{split}
\end{equation}

Moreover, by taking into account the trivial relation
$(|\nabla u|-4 |A|\,|u|)^2\geq 0$, we deduce
\begin{equation}\label{4products}
4 |A|\,|u|\,|\nabla u|\leq \frac{1}{2}|\nabla u|^2+8|A|^2 \, |u|^2\,.\\
\end{equation}
Then, by combining \eqref{4products} and \eqref{eq:norm2}, we get
\begin{equation}\label{eq:H1ABound}
|\nabla_A u|^2\geq \frac{1}{2}|\nabla u|^2-7|A|^2 \, |u|^2\,,\\
\end{equation}
or, conversely,
\begin{equation}\label{eq:H1Bound}
|\nabla u|^2\leq 2|\nabla_A u|^2+14|A|^2 \, |u|^2\,.\\
\end{equation}

As far as no ambiguity arises, we will not distinguish, in notation, between the magnetic field and the skew-symmetric $N\times N$ matrix-valued function that represents it in the Euclidean space, and, respectively, between the magnetic potential and its representation as a vector field. In this context we also allow differentiation of scalar functions to be denoted, interchangeably as $\nabla$ and as $d$, depending if we view the resulting value as a vector field or as a 1-form.

When dealing with a magnetic potential $A$ we shall define $\dot H_A^{1,2}(\R^N)$ as the completion of $C_0^\infty$ with respect to the norm $\|u\|_A=(E_A(u))^\frac12$
(note that $\dot H^{1,2}_A(\R^N)$ is a space of measurable functions whenever $N>2$ or $N=2$ and $dA\not\equiv 0$, see e.g.\ \cite{Enstedt}), and the space 
$H_A^{1,2}(\R^N)$ as the intersection $\dot H_A^{1,2}(\R^N)\cap L^2(\R^N)$ (equipped with the standard intersection norm). Note that $|u|^2$ in this model has the meaning of the probability distribution of a particle in the space, and it doesn't change if $u$ is also affected by the multiplication with $e^{\i\varphi}$ (on the other side, as already remarked, the magnetic potential $A$ and $A+\nabla\varphi$ give rise to the same magnetic field).
Note also that, when $A\equiv 0$, we obtain the corresponding usual Sobolev spaces $\dot H^{1,2}(\R^N)$ and $H^{1,2}(\R^N)$ respectively.
Moreover, in such a case, the re-phasing function $\varphi_y$ in \eqref{eq:m-shifts} can be set to be constant for all $y$.
We shall later normalize the magnetic shifts $g_y$ (see \eqref{eq:givenpoint} and \eqref{zero} below) in such a way that they reduce to Euclidean shifts when the magnetic field is zero.
%
%

Setting, as usual, $2^*=\begin{cases}
\frac{2N}{N-2}, &N>2\\
\infty, &N=2 \end{cases}$
the critical Sobolev exponent, we have the following continuous embeddings
$H^{1,2}(\R^N)\hookrightarrow L^p(\R^N)$ for every $p\in[2,2^*)$,
 and $H^{1,2}(\R^N) \hookrightarrow \dot H^{1,2}(\R^N)\hookrightarrow  L^{2^*}(\R^N)$ when $N>2$. Furthermore, from the well-known {\em{diamagnetic inequality}} 
\begin{equation}\label{eq:diamag}
|\nabla_A u(x)|\ge \left|\nabla|u(x)|\right|,
\end{equation}
we deduce that
\begin{equation*}
\label{eq:ineq}
E_A(u)\ge E_0(|u|)\ge C\|u\|_{2^*}^2, \quad \mbox{ if } N>2,
\end{equation*}
and, as a consequence, the continuous embeddings
\begin{equation*}
\label{eq:emb}
\dot H_A^{1,2}(\R^N)\hookrightarrow L^{2^*}(\R^N)\quad \mbox{ and }\quad H_A^{1,2}(\R^N)\hookrightarrow  L^{p}(\R^N) \;\forall p\in[2,2^*),
\end{equation*}
(the latter one extends by analogous argument to the case $N=2$).

Note that, when the magnetic field $B$ is lattice periodic, i.e.\ when
\begin{equation}\label{eq:lattice}
B(\cdot-y)=B(\cdot)\qquad \mbox{ for all } y\in\Z^N,
\end{equation}
we have, in terms of a fixed magnetic potential $A$,
$$d(A(\cdot-y)-A(\cdot))=0 \qquad \mbox{ for all } y\in \Z^N.$$
Therefore, 
\begin{equation}
\label{eq:phiy}
\forall y\in\Z^N\; \exists \varphi_y \mbox{ s.t. } A(\cdot -y)=A(\cdot)+\nabla\varphi_y(\cdot).
\end{equation}
From this we can derive the following relation which will be used in the Appendix. Namely for any $y_1,\,y_2 \in\Z^N$ there exists a constant $\gamma(y_1, y_2)\in\R$ such that
\begin{equation}\label{eq:etaprod}
\varphi_{y_1+y_2}=\varphi_{ y_1}(\cdot-y_2)+\varphi_{y_2 }+\gamma(y_1, y_2),
\end{equation}
(indeed, derivatives of the left and of the right hand side coincide by \eqref{eq:phiy}). 
Thus, for every $y\in\Z^N$, we can define, by \eqref{eq:m-shifts}, a suitable magnetic shift $g_y$ such that, (still denoting by $g_y$ its extension to vector valued functions) we have the following commutation law
\begin{equation*}
\nabla_A(g_y u)=g_y(\nabla_A u),
\end{equation*}
where, actually, for all $x\in \R^N$,
\begin{equation*}
\label{eq:gnabla}
g_y(\nabla_A u)(x) \eqdef e^{\i\varphi_y(x)}\nabla_A u (x-y)=
e^{\i\varphi_y(x)}(\nabla u(x-y)+ \i A(x-y)u(x-y)).
\end{equation*}
Indeed, by taking into account (\eqref{eq:m-shifts}, \eqref{eq:nablaA} and) \eqref{eq:phiy}, we have
\begin{equation}\label{eq:proofcommutation}
\begin{split}
\nabla_A(g_y u)(x) & =e^{\i\varphi_y(x)}\left( \nabla u (x-y)+\i (\nabla \varphi_y(x) + A(x)) u(x-y)\right)\\
& = e^{\i\varphi_y(x)}\left( \nabla u (x-y)+\i A(x-y) u(x-y)\right)\\
&= g_y(\nabla_A u)(x).
\end{split}
\end{equation}
As a consequence, we have that the magnetic shifts $g_y$ (given by \eqref{eq:m-shifts} where 
$\varphi_y$ satisfies \eqref{eq:phiy}) are, for all $y\in \Z^N$, isometries on
$\dot H^{1,2}_A(\R^N)$, indeed, for all $u, v \in \dot H^{1,2}_A(\R^N)$, we have
\begin{equation}\label{eq:latticescalar}
\pairing{g_y u}{g_y v}_A=\pairing{u}{v}_A.
\end{equation}
%
%
Indeed, by taking into account \eqref{eq:proofcommutation} and since $\varphi_y$ is a real valued function, we get
%
that
\begin{equation}\label{eq:proofinvariance}
\begin{split}
\pairing{g_y u}{g_y v}_A & = \int_{\R^N}\nabla_A (g_y u) \overline{\nabla_A (g_y v)}\dx x\\
& = \int_{\R^N}g_y(\nabla_A u) \,\overline{g_y(\nabla_A v)}\dx x\\
& = \int_{\R^N} \nabla_A u(x-y) \, \overline{\nabla_A v(x-y)}\dx x= \pairing{u}{v}_A.\\
\end{split}
\end{equation}
As a consequence of \eqref{eq:latticescalar} we have
\begin{equation}
\label{eq:shifts}
E_A(g_y u)=E_A(e^{\i\varphi_y(\cdot)}u(\cdot-y))=E_{A}(u)\qquad \mbox{ for all }y\in \Z^N.
\end{equation}
\vskip2mm
The following lemma can be found elsewhere in literature (e.g. \cite{ArioliSzulkin,Weidl}) but we include our version here for the consistency of the paper.
\begin{lemma} \label{lem:localmag} 
	Let $\Omega\subset\R^N$ be an open bounded set  with piecewise $C^1$-boundary and let $A\in C(\Omega,\Lambda_1)$. Then, for any $\lambda>0$, there exist positive constants $C_1$ and $C_2$ such that for all $u\in H^{1,2}_A(\R^N)$:
\begin{equation}
	\label{eq:localmag}
	C_1\int_\Omega(|\nabla u|^2+|u|^2)\dx{x}\le \int_\Omega(|\nabla_A u|^2+\lambda|u|^2)\dx{x} \le  C_2\int_\Omega(|\nabla u|^2+|u|^2)\dx{x}.
\end{equation}	 
\end{lemma}
\begin{proof}
	First note that there exist positive constants $\Lambda$, $C_1$ and $C_2$  such that 
	\begin{equation}
	\label{eq:localmag0}
	C_1\int_\Omega(|\nabla u|^2+|u|^2)\dx{x}\le \int_\Omega(|\nabla_A u|^2+\Lambda|u|^2)\dx{x} \le  C_2\int_\Omega(|\nabla u|^2+|u|^2)\dx{x}
	\end{equation}
for all $u\in H^{1,2}_A(\R^N)$.		 
Indeed, the right inequality is elementary and the left one easily follows by \eqref{eq:H1Bound}.
Let us now recall that by compactness of local Sobolev embeddings, for any $\epsilon>0$ there exists a subspace $E_\epsilon>0$ of finite codimension such that 
\[
\int_\Omega |u|^2\dx{x}\le \epsilon \int_\Omega(|\nabla u|^2+|u|^2)\dx{x}, \;u\in E_\epsilon. 
\]
Then \eqref{eq:localmag} restricted to $E_\epsilon$ is immediate from \eqref{eq:localmag0}. Furthermore, since all norms are equivalent on finite-dimensional spaces, \eqref{eq:localmag} holds also on a complement of $E_\epsilon$, and then, by the triangle inequality, it holds for all  $u\in H^{1,2}_A(\R^N)$.
\end{proof}	

\begin{corollary}
	\label{cor:Frechet}
For any $A\in C_{\mathrm{loc}}(\R^N,\Lambda_1)$ the
Fr\'echet spaces $H^{1,2}_{A,\mathrm{loc}}(\R^N)$ and  $H^{1,2}_{\mathrm{loc}}(\R^N)$ coincide.
\end{corollary}
\section{Energy at infinity}
Let $N\ge 2$ and let $\Lambda_1$ and $\Lambda_2$ denote the linear spaces, respectively, of 1-forms on $\R^N$ and of antisymmetric 2-forms on $\R^N\times\R^N$.
Let, for $\a\in(0,1]$, 
$C^{0,\a}$  and $C^{1,\a}$ denote, respectively, the space of H\"older-continuous functions  
and the space of functions with H\"older-continuous first derivative, with H\"older exponent $\a$.
Given a magnetic field $B\in C^{0,\a}(\R^N,\Lambda_2)$, we will say that  $A\in C(\R^N,\Lambda_1)$ is a magnetic potential of $B$ if $B=dA$ in the sense of  weak differentiation.
Note that, when $B$ is bounded, we can assume that $A(x)=O(|x|)$ near infinity (we are using the Landau $O$ symbol to denote functions of the same order). In particular, when $B$ is a constant magnetic field, $A$ is linear in $x$. Therefore, generally, we cannot define, for a given function $u\in H^{1,2}_A(\R^N)$, the limit of the magnetic energy functional \eqref{eq:hamiltonian} when the function $u$ is shifted at infinity like in the non-magnetic case (e.g.\ \cite{Lions1}), by taking $\lim_{|y|\to\infty} E_0(u(\cdot+y))$.
This difficulty is easily overcome in the case of a lattice-periodic magnetic field. Indeed,  by \eqref{eq:shifts}, the functional $E_A\circ g_{y_k}$ remains equal to $E_A$ for any sequence
$y_k\in\Z^N$.

Below, in Lemma~\ref{lem:phase}, we construct a re-phasing function $\varphi_y$, $y\in\R^N$, such that the corrected magnetic potential
\begin{equation}
\label{eq:correctedA}
A_y=A+\nabla\varphi_y,
\end{equation}
which still corresponds to the same magnetic field $B=d A$, 
is bounded on balls $B_r(y)$ by a constant which is dependent on $r$ but is uniform in $y$.
Without periodicity assumptions on $B$, $\varphi_y$ will not satisfy \eqref{eq:phiy}, but the difference between the left and right hand side of \eqref{eq:phiy} will be still controlled.
%
This allows us to define  magnetic shifts as in \eqref{eq:m-shifts} and to give a definition of the limit value of the energy of a function $u$ subjected to a sequence of magnetic shifts related to a diverging sequence of points in $\R^N$.

In the next lemma we shall make use of the following notation:
$\forall x\in \R^N, x=(x_1,\dots,x_N)$ and $\forall i\in \{1,\dots,N\}$ we shall set
\begin{equation*}
\label{eq:X}
\underleftarrow{X}_i:= (x_1,\dots, x_i)\in \R^{i}
\quad \mbox{ and }\quad
\overrightarrow{X}^i:= (x_i,\dots, x_N)\in \R^{N-i+1},
\end{equation*}
so that, for all $i\in \{1,\dots,N-1\}$,
\begin{equation*}
\underleftarrow{X}_N=\overrightarrow{X}^1=x =(\underleftarrow{X}_i,\overrightarrow{X}^{i+1}),
\end{equation*}
while 
\begin{equation*}
\underleftarrow{X}_1=x_1 \quad \mbox{ and } \quad \overrightarrow{X}^N=x_N. 
\end{equation*}

\begin{lemma}\label{lem:phase}
Let $A:=(A_1,\dots,A_N)\in C_{\mathrm{loc}}(\R^N,\Lambda_1)$. Then, for all $y=(y_1,\dots,y_N)\in\R^N$, there exists a function $\varphi_y\in C(\R^N)$ such that for all $x=(x_1,\dots,x_N)$ the following $N$ partial derivatives of $\varphi_y$ exist and are expressed by means of the coordinate functions of the magnetic potential as follows: 
\begin{equation}
\label{eq:equalities}
\begin{split}
\partial_1\,\varphi_y(\overrightarrow{X}^1)&= - A_1(\overrightarrow{X}^1),\\
\partial_2\,\varphi_y(\underleftarrow{Y}_1,\overrightarrow{X}^2)&= 
- A_2(\underleftarrow{Y}_1,\overrightarrow{X}^2),\\
\partial_3\,\varphi_y(\underleftarrow{Y}_2,\overrightarrow{X}^3)&= 
- A_3(\underleftarrow{Y}_2,\overrightarrow{X}^3),\\
&\dots \\
\partial_N\,\varphi_y(\underleftarrow{Y}_{N-1},\overrightarrow{X}^{N})&=
- A_N(\underleftarrow{Y}_{N-1},\overrightarrow{X}^{N}).
\end{split}
\end{equation}
In particular
\begin{equation}
\label{eq:A=nabla}
\nabla\varphi_y(y)=-A(y).
\end{equation}
\end{lemma}
\begin{proof}
To any fixed $x=(x_1,\dots,x_N)\in \R^N$ we shall assign a unique real number, to be denoted as $\psi(x)$, first by arbitrarily fixing the value $\psi(y)$  
at $y=(y_1,\dots,y_N)$, then allowing successively to vary, in the descending order, each of the variables $y_i$, $i=N,N-1,\dots,1$.

We start with $\psi(y_1,\dots,y_{N-1},x_N)$:
\[
\psi(\underleftarrow{Y}_{N-1},\overrightarrow{X}^N)=\psi(\underleftarrow{Y}_{N-1},y_N)+\int_{y_N}^{x_N}A_N(\underleftarrow{Y}_{N-1},t)\dx{t},
\]
we can recursively define, for decreasing $m=N-1,\dots,2$,
\[
\psi(\underleftarrow{Y}_{m-1},x_m,\overrightarrow{X}^{m+1})=
\psi(\underleftarrow{Y}_{m-1},y_m,\overrightarrow{X}^{m+1})+\int_{y_{m}}^{x_{m}}A_{m}(\underleftarrow{Y}_{m-1},t,\overrightarrow{X}^{m+1})\dx{t}.
\]
Finally, for $m=1$, we set
\[
\psi(x)=\psi(x_1,\overrightarrow{X}^2)=\psi(y_1,\overrightarrow{X}^2)+\int_{y_1}^{x_1}A_1(t,\overrightarrow{X}^2)\dx{t}. 
\]
Then, the claim follows by setting $\varphi_y\eqdef -\psi$. Functions $\varphi_y$, $y\in\R^N$, satisfy the required properties by construction. Each of them is, of course, defined up to the arbitrary chosen initial value $\varphi_y(y)=-\psi(\underleftarrow{Y}_N)=-\psi(y)$.
\end{proof}
Note that, 
by \eqref{eq:A=nabla}, no correction is needed in the point $y$, i.e.\ the value of the corrected potential $A_y$
defined by \eqref{eq:correctedA} is zero at $y$, i.e.\
\begin{equation}
\label{eq:Ay}
A_y(y)=A(y)+\nabla \varphi_y(y)=0.
\end{equation}
Moreover, by taking into account \eqref{eq:correctedA} and \eqref{eq:equalities}, we deduce that
\begin{equation}
\label{eq:oncomponents}
(A_y)_n (\underleftarrow{Y}_{n-1},\cdot)=A_n(\underleftarrow{Y}_{n-1},\cdot)+
\partial_n \varphi_y(\underleftarrow{Y}_{n-1},\cdot)\equiv 0, \quad \forall n\in\{1,\dots,N\},
\end{equation}
(where we use the convention $\underleftarrow{Y}_{0}=\emptyset$ so that the above relation means, for $n=1$, $(A_y)_1\equiv 0$).
\begin{remark}
\label{rem:y=0} 
In what follows the corrected magnetic potential $A_y$ is defined by \eqref{eq:correctedA} where $\varphi_y$ is the function provided by Lemma~\ref{lem:phase} if $y\neq 0$, while for $y=0$ we set $\varphi_0=0$, so that $A_0$, defined by \eqref{eq:correctedA}, equals $A$. Consequently, the magnetic shift $g_0$ in the definition below becomes the identity operator.
\end{remark}
We now define magnetic shifts (with approximate re-phasing) as follows.
\begin{definition}
\label{def:g}
Let $B=d A$ be a magnetic field with $A\in C^1_{\mathrm{loc}}(\R^N, \Lambda_1)$. 
Any map 
\begin{equation}
\label{eq:g}
g_y := u\mapsto  e^{\i \varphi_y(\cdot)} u(\cdot-y) \quad \forall u\in C_0^1(\R^N),\;y\in\R^N,
\end{equation}
where the $C^1(\R^N)$ function $\varphi_y$ is provided by Lemma \ref{lem:phase2} when $y\neq 0$ or is the null function when $y=0$ (see Remark \ref{rem:y=0}), is called \emph{a magnetic shift} (relative to the magnetic field $B$) determined by the vector $y$.
\end{definition}

Arguing as in \eqref{eq:proofcommutation}, we get, by using 
\eqref{eq:correctedA} (instead of \eqref{eq:phiy}), that 
\begin{equation}\label{eq:commutation}
\nabla_A(g_y u)= g_y(\nabla_{A_y(\cdot +y)}u)\quad
\mbox{ and }\quad
(g_y)^{-1}(\nabla_A u)= \nabla_{A_y(\cdot +y)}((g_y)^{-1}u).
\end{equation}
So, arguing as in \eqref{eq:proofinvariance}, we get
\begin{equation*}\label{eq:scalargauge}
\pairing{g_y u}{g_y v}_A= \pairing{u}{v}_{A_y(\cdot +y)}\quad
\mbox{ and }\quad
\pairing{(g_y)^{-1}u}{(g_y)^{-1}v}_{A_y(\cdot +y)}=\pairing{u}{v}_A,
\end{equation*}
and, in particular, that 
\begin{equation}\label{eq:EAphasing}
E_A(g_y u) =
E_{A_y(\cdot+y)}(u)
\quad
\mbox{ and }\quad
E_A(u)=E_{A_y(\cdot +y)}((g_y)^{-1}u).\\
\end{equation}
%
%
In next lemma we prove that, 
given a bounded continuous magnetic field~$B$, all corrected magnetic potentials $A_y$ (which vanish at $y$) satisfy a bound of linear type.
%
\begin{lemma}
\label{lem:phase2}
Let $B=d A$ be a bounded magnetic field with $A\in C^1_\mathrm{loc}(\R^N,\Lambda_1)$. Let the functions $\varphi_y$, $y\in\R^N$ be as in Remark \ref{rem:y=0} and let $A_y$ be the corrected potential defined by \eqref{eq:correctedA}.
Then for all $y\in\R^N$,
\begin{equation}
\label{eq:Aineq}
|A_y(x)|\le \|B\|_\infty\,|x-y|\quad \forall x\in\R^N,
\end{equation}
where 
\begin{equation}
\label{eq:Binfty}
\|B\|_\infty^2\eqdef \sum_{n=1}^N\sum_{m<n}\|B_{mn}\|_\infty^2,
\end{equation}
and 
\begin{equation}
\label{eq:Bmn}
B_{mn}=\partial_n A_m-\partial_m A_n \quad \forall m,n\in \{1,\dots,N\}.
\end{equation}

\end{lemma}
\begin{proof} Consider, for a fixed $y\in \R^N$, the form $A_y$ given by \eqref{eq:correctedA} and Lemma~\ref{lem:phase}.
We shall prove \eqref{eq:Binfty} by showing first that 
(with the convention that the sum over an empty set of indices is zero) the following estimate holds true
\begin{equation}
 \label{Ayyy}
|(A_y)_n(x)|\le\sum_{m=1}^{n-1}\|B_{mn}\|_\infty |x_m-y_m|,\quad x\in\R^N.
\end{equation} 
By using \eqref{eq:oncomponents} (with $n$ replaced by $n-1$) and \eqref{eq:Bmn}, we have
\begin{equation*}
\begin{split}
(A_y)_{n}(\underleftarrow{Y}_{n-2},\overrightarrow{X}^{n-1})&=(A_y)_{n}(\underleftarrow{Y}_{n},\overrightarrow{X}^{n+1})+
\int_{y_{n-1}}^{x_{n-1}}\partial_{n-1}(A_y)_{n}(\underleftarrow{Y}_{n-2},t,\overrightarrow{X}^{n})\dx{t}\\
&= \int_{y_{n-1}}^{x_{n-1}}\partial_{n-1} (A_y)_{n}(\underleftarrow{Y}_{n-2},t,\overrightarrow{X}^{n})\dx{t}\\
& \le -\int_{y_{n-1}}^{x_{n-1}}\partial_n A_{n-1}(\underleftarrow{Y}_{n-2},t,\overrightarrow{X}^{n})\dx{t} \\
&\quad + \|B_{n-1\;n}\|_\infty|x_{n-1}-y_{n-1}|\\
& =\|B_{n-1\;n}\|_\infty|x_{n-1}-y_{n-1}|.
\end{split}
\end{equation*}
This argument can be repeated for $-(A_y)_{n}(\underleftarrow{Y}_{n-1},\overrightarrow{X}^n)$ yielding the same upper bound for $|(A_y)_{n}(\underleftarrow{Y}_{n-1},\overrightarrow{X}^n)|$.

Repeating the same argument for $(A_y)_{n}(\underleftarrow{Y}_{n-3},\overrightarrow{X}^{n-2})$ while using integration with respect to $x_{n-2}$, we get
\begin{equation*}
\begin{split}
|(A_y)_{n}(\underleftarrow{Y}_{n-3},\overrightarrow{X}^{n-2})|& \le 
|(A_y)_{n}(\underleftarrow{Y}_{n-1},\overrightarrow{X}^n)|
+ \|B_{n-2\;n}\|_\infty|x_{n-2}-y_{n-2}|\\
& \le \|B_{n-2\;n}\|_\infty|x_{n-2}-y_{n-2}| +\|B_{n-1\;n}\|_\infty|x_{n-1}-y_{n-1}|.
\end{split}
\end{equation*}
Iterating the same estimate $n-1$ times one gets \eqref{Ayyy} and then \eqref{eq:Binfty} follows from Cauchy inequality. Indeed, setting $a_n\eqdef|x_n-y_n|$ and $b_{mn}\eqdef\|B_{mn}\|_\infty$, we get by \eqref{eq:Binfty},
\begin{equation*}
\begin{split}
\sum_{n=1}^N|(A_y)_n|^2& \le \sum_{n=1}^N \left(\sum_{m<n}b_{mn}a_m\right)^2\le
\sum_{n=1}^N \left(\sum_{m<n}b_{mn}^2 \sum_{i<n} a_i^2\right)\\
& \le 
\left(\sum_{i=1}^N a_i^2\right)
\sum_{n=1}^N \sum_{m<n}b_{mn}^2= 
 |x-y|^2\|B\|_\infty^2.
\end{split}
\end{equation*}
\end{proof}

In what follows $\dot C^1(\R^N)$ denotes the space of functions with uniformly bounded derivatives and  $\dot C^{0,1}(\R^N)$ denotes the space of functions satisfying $|f(x)-f(y)|\le C|x-y|$ for all $x,y\in\R^N$.
\begin{remark}\label{rem:Lipschitz}
	Let $Y=(y_k)_{k\in\N}$ be a diverging sequence in $R^N$. If $A\in \dot C ^1(\R^N,\Lambda_1)$, then applying Arzel\`a-Ascoli theorem 
	to the sequence $(A_{y_k}(\cdot + y_k))_{k\in\N}$,
	we get a renamed subsequence such that 
	$(A_{y_k}(\cdot+y_k))_{k\in\N}$ and $(dA_{y_k}(\cdot+y_k))_{k\in\N}$ converges, uniformly on bounded sets, respectively to some Lipschitz  function $A^{(Y)}_\infty\in \dot C^{0,1}(\R^N,\Lambda_1)$ and a bounded function $B^{(Y)}_\infty=dA^{(Y)}_\infty\in L^\infty_\mathrm{loc}(\R^N,\Lambda_1)$. 
\end{remark}
\begin{definition}
\label{def:AY}
Let $Y= (y_k)_{k\in\N}$ be a diverging sequence in $\R^N$,
and consider the renamed subsequence $(y_k)_{k\in\N}$ and the associated magnetic potential $A^{(Y)}_\infty\in C^{0,1}_\mathrm{loc}(\R^N,\Lambda_1)$ described in Remark~\ref{rem:Lipschitz}. Then, the corresponding functional $E_{A^{(Y)}_\infty}$ (see \eqref{eq:hamiltonian}) will be called \emph{the energy at infinity relative to the renamed subsequence $Y=(y_k)_{k\in\N}$}, corresponding to the \emph{magnetic field at infinity} $B^{(Y)}_\infty= dA^{(Y)}_\infty$. 
\end{definition}

We end this section by defining sequences without concentrations and, to this aim, we shall fix a discretization which shall play the role of $\Z^N$ in the periodic setting.
\begin{definition}
A set $\Xi\subset \R^N$ is called a discretization of $\R^N$ if $0\in \Xi$,
\[
\inf_{x,y\in \Xi, x\neq y} |x-y|>0,
\]
and if, for some $\rho>0$, $\{B_\rho(x)\}_{x\in\Xi}$ is a covering of $\R^N$ of uniformly finite multiplicity.

\end{definition}
Note that if $\Xi$  is a discretization of $\R^N$ then, for any $R>\rho$, the covering $\{B_R(x)\}_{x\in\Xi}$ is still of uniformly finite multiplicity. A trivial example for a discretization of $\R^N$ is $\Z^N$.
\begin{definition}
\label{def:G}
Let $\Xi\subset \R^N$ be a discretization of $\R^N$ and
let 
\begin{equation}
\mathcal G_\Xi \eqdef  \{g_y: u\mapsto e^{\i\varphi_y(\cdot)}u(\cdot-y) \; |\; y\in\Xi \mbox{ and }\varphi_y \mbox{ as in Remark } \eqref{rem:y=0}\}\,.
\end{equation}
We shall say that a sequence 
$(u_k)_{k\in\N}$ in $H^{1,2}_A(\R^N)$ $\mathcal{G}_/Xi$-converges to zero or is $\mathcal{G}_/Xi$-infinitesimal, writing $u_k\cw 0$,
if 
for every sequence
$(g_k)_{k\in\N}\subset \mathcal{G}_\Xi$, $g^{-1}_k u_k\rightharpoonup 0$ in $H^{1,2}_{\mathrm{loc}}(\R^N)$.
\end{definition}
Note that, by Corollary~\ref{cor:Frechet}, the weak convergence above is also in $H^{1,2}_{A,\mathrm{loc}}(\R^N)$ for any $A\in C_{\loc}(\R^N,\Lambda_1)$ (including $A\equiv 0$).
\begin{remark}
	It is not difficult to show that $\mathcal{G}_\Xi$-convergence is independent of discretization $\Xi$ (cf. Remark~\ref{rem:acc-shifts} below for the non-magnetic case), but will not use this property in any subsequent argument.	
\end{remark}
%
%
%
%
%
%
%
%
%
\section{A vanishing lemma of cocompactness type}\label{cocompactness}
In order to describe defect of compactness for bounded sequences in  $H_A^{1,2}(\R^N)$ we will first characterize the behavior of sequences without concentrations, i.e. of $\mathcal{G}$-infinitesimal sequences (see Definition \ref{def:G}). We start with a statement about local boundedness.
\begin{lemma}\label{lem:localbound}
	Let $N\ge 3$, $A\in C_{\mathrm{loc}}(\R^N,\Lambda_1)$, let $(u_k)_{k\in\N}$ be a bounded sequence in $H_A^{1,2}(\R^N)$. Then, for any sequence $(y_k)_{k\in\N}$ in $\R^N$,
the sequence $(g^{-1}_{y_k} u_k)_{k\in\N}$ is bounded in $H^{1,2}_\mathrm{loc}(\R^N)$.
\end{lemma}
\begin{proof}
	The local $L^2$-bound follows from the Sobolev inequality and the diamagnetic inequality. 
	Indeed, fixed $R>0$, there exists a constant $S(R)$, such that 
		\begin{equation}\label{eq:L2bound}
		\begin{split}
	\int_{B_R(0)}|g^{-1}_{y_k}u_k|^2\dx{x} & =\int_{B_R(y_k)}|u_k|^2\dx{x}\\
 & \le
	S(R) \left(\int_{B_R(y_k)}|u_k|^{2^*}\dx{x}\right)^\frac{2}{2^*}\\
	& \le
	S(R) \left(\int_{\R^N}|u_k|^{2^*}\dx{x}\right)^\frac{2}{2^*}\\
 & \le
	\widetilde{S}(R)\int_{\R^N}|\nabla|u_k||^2\dx{x}\le \widetilde{S}(R) E_A(u_k). 
	\end{split}
	\end{equation}
Now, by taking into account \eqref{eq:EAphasing} we get, for any $R>0$, that
$$
E_A(u_k)=E_{A_{y_k}(\cdot+y_k)}(g^{-1}_{y_k} u_k) \geq
\int_{B_R(0)}|\nabla_{A_{y_k}(\cdot+y_k)} g^{-1}_{y_k} u_k|^2 \dx{x},
$$
and, by applying \eqref{eq:H1ABound} (with $A$ and $u$ replaced by $A_{y_k}(\cdot+y_k)$ and $g^{-1}_{y_k} u_k$ respectively), we deduce that
\begin{equation*}
\begin{split}
E_A(u_k)& \geq
\frac{1}{2}\int_{B_R(0)}|\nabla g^{-1}_{y_k} u_k|^2 \dx{x} - 
7 \int_{B_R(0)}|A|^2\, |g^{-1}_{y_k} u_k|^2 \dx{x}\\
& \geq 
\frac{1}{2}\int_{B_R(0)}|\nabla g^{-1}_{y_k} u_k|^2 \dx{x} -
7 C(R) \int_{B_R(0)}|g^{-1}_{y_k} u_k|^2 \dx{x} \\
& \geq 
\frac{1}{2}\int_{B_R(0)}|\nabla g^{-1}_{y_k} u_k|^2 \dx{x} - 7 C(R) \widetilde{S}(R)E_A(u_k)\,,\\
\end{split}
\end{equation*}
where, to get last two inequalities we have used \eqref{eq:Aineq} and \eqref{eq:L2bound}.
So, by combining the above inequalities, we get that for a suitable positive constant
$\widetilde{C}(R)$
$$
\int_{B_R(0)}|\nabla g^{-1}_{y_k} u_k|^2 \dx{x}\leq \widetilde{C}(R) E_A(u_k)\,.
$$
\end{proof}
%
%
\begin{lemma}\label{lem:vanish}
For any bounded sequence $(u_k)_{k\in\N}$ in $H_A^{1,2}(\R^N)$, $A\in C_{\mathrm{loc}}(\R^N,\Lambda_1)$, the following implication holds true 
$$u_k\cw 0 \quad \Rightarrow \quad u_k\to 0 \mbox{ in }L^p(\R^N) \, \mbox{ for any } p\in(2,2^*).$$
\end{lemma}
\begin{proof}
Let $p\in(2,2^*)$, $\Xi$ be a discretization of $\R^N$, and fix $y\in \Xi$. By combining the Sobolev inequality on an open ball $B_\rho(y)$, with the diamagnetic inequality \eqref{eq:diamag}, we get, with some constant $C>0$ independent of $y$, 
\begin{equation*}
\begin{split}
\int_{B_\rho(y)}|u_k|^p\dx{x} &= \left(\int_{B_\rho(y)}|u_k|^p\dx{x} \right)^{2/p} \left(\int_{B_\rho(y)}|u_k|^p\dx{x}\right)^{1-2/p}\\ 
& \le 
C\int_{B_\rho(y)} (|\nabla |u_k||^2+|u_k|^2)\dx{x}\left(\int_{B_\rho(y)}|u_k|^p\dx{x}\right)^{1-2/p}\\
& \le 
C\int_{B_\rho(y)}(|\nabla_A u_k|^2+|u_k|^2)\dx{x}\left(\int_{B_\rho(y)}|u_k|^p \dx{x}\right)^{1-2/p}\\
& \le 
C\int_{B_\rho(y)}(|\nabla_A u_k|^2+|u_k|^2)\dx{x}\; \sup_{z\in \Xi}\left(\int_{B_\rho(z)}|u_k|^p \dx{x}\right)^{1-2/p}\,.\\
\end{split}
\end{equation*}
Adding over all $y\in\Xi$, and taking into account that the covering $\{B_\rho(y)\}_{y\in\Xi}$ has uniform finite multiplicity, we have, for a suitable choice of $y_k\in\Xi$ (and with a change of variable), that 
\begin{equation*}
\begin{split}
\int_{\R^N}|u_k|^p\dx{x}
& \leq  C\|u_k\|_{H^{1,2}_A(\R^N)}\sup_{y\in\Xi} \left(\int_{B_\rho(y)}|u_k|^p\dx{x}\right)^{1-2/p}\\
& \leq 2C\|u_k\|_{H^{1,2}_A(\R^N)}
\left(\int_{B_\rho(y_k)}|u_k|^p\dx{x}\right)^{1-2/p}\\
& = 2C\|u_k\|_{H^{1,2}_A(\R^N)}
\left(\int_{B_\rho(0)}|u_k(x+y_k)|^p\dx{x}\right)^{1-2/p}\,.
\end{split}
\end{equation*}
Since, by assumption, $(g^{-1}_{y_k}u_k)_{k\in\N}$ converges weakly to zero in
$H^{1,2}(B_\rho(0))$ which is compactly embedded in $L^p(B_\rho(0))$, we get that $u_k(\cdot+y_k)\to 0$ in 
$L^p(B_\rho(0))$ and, by the chain of inequalities above, that $(u_k)_{k\in\N}$ vanishes in $L^p(\R^N)$.
\end{proof}

Note that, despite Corollary \ref{cor:Frechet}, spaces $H_A^{1,2}(\R^N)$ and $\dot H_A^{1,2}(\R^N)$ do not generally coincide, in particular, they are distinct if $B=dA=0$. In general, unless $B$ is identically zero in $\R^N$ and $N=2$, there exists a positive function $W$ such that $E_A(u)\ge\int_{\R^N}|u|^2W(x)\dx{x}$ (see e.g. \cite{Enstedt}), but $W$ may be not necessarily bounded from below. If $dA$ is lattice-periodic, it is easy to show, via partition of unity, that 
$E_A(u)\ge C\int_{\R^N}|u|^2\dx{x}$ so that, if 
$u\in \dot H_A^{1,2}(\R^N)$, then $u\in L^2(\R^N)$, i.e.\ $H_A^{1,2}(\R^N)=\dot H_A^{1,2}(\R^N)$.
For sequences bounded in $\dot H_A^{1,2}(\R^N)$, but not bounded in $L^2(\R^N)$,
one can prove an analog of Lemma~\ref{lem:vanish}.

\begin{lemma}\label{lem:vanish2}
Assume that $A\in C_{\mathrm{loc}}(\R^N,\Lambda_1)$ and that for some $r>0$, $B=dA\not \equiv 0$ on any ball $B_r(x)$, $x\in\R^N$ (so that there is no sequence of balls of radius going to infinity where $dA\equiv 0$). Then, for any $p\in(2,2^*)$, there is a positive continuous bounded function $W_p$ defined on $\R^N$ such that 
\begin{equation*}
\label{eq:weak2strong}
u_k\cw 0\mbox{ in } \dot H_\mathrm{loc}^{1,2}(\R^N) \quad \Rightarrow \quad u_k\to 0 \mbox{ in } L^p(\R^N,W_p\dx{x}).
\end{equation*}
\end{lemma}
\begin{proof}
The proof follows the same argument used for Lemma~\ref{lem:vanish} with the following elementary modifications. The radius of the balls $B_\rho(y)$, $y\in\Xi$, have to be changed to 
$R=\max(2\rho, r)$, so that Sobolev inequality 
on the ball $B_R(y)$, $y\in \Xi$ (with a constant \emph{dependent} on $y$) will follow from the assumption $dA\not\equiv 0$ on $B_r(y)$. This constant, after summation over $\Xi$ (recall that the covering by $B_R(y)$ remains of finite multiplicity) will produce a piecewise-constant weight in the $L^p$-norm, which can be replaced by a continuous function $W_p$.
\end{proof}
\section{Profile decomposition}\label{PD}
The purpose of this section is to prove the profile decomposition stated in Theorem \ref{thm:MagneticPD} below with the related energy bounds stated in Theorem \ref{thm:energies}.

When we shall deal with sequences 
$Y^{(n)}\eqdef (y_k^{(n)})_{k\in\N}$ in $\R^N$ depending on a parameter $n\in \N$, we will abbreviate the notation for the magnetic potentials at infinity $A^{(Y^{(n)})}_\infty$ as $A^{(n)}_\infty$ (and the corresponding magnetic field as $B^{(n)}_\infty$).
Moreover, for any $n$ and for any $k$, $g_{y^{(n)}_k}$ will represent the magnetic shift defined by \eqref{eq:g} with $y=y^{(n)}_k$, and we will use the abbreviated notation
\begin{equation}\label{eq:short}
g^{(n)}_k\eqdef g_{y^{(n)}_k}\quad \mbox{ as well as }
\quad A^{(n)}_k \eqdef A_{y^{(n)}_k}(\cdot +y^{(n)}_k)\,.
\end{equation}
We warn the reader that we shall always reserve the index $n=0$ to the trivial sequence $Y^{(0)}\eqdef (y_k^{(0)})_{k\in\N}=(0)_{k\in\N}$ (which, of course is not diverging) and always set  
$A^{(0)}_k=A^{(0)}_\infty=A$ (indeed, the magnetic shifts
$g^{(0)}_k=g_{y_k^{(0)}}$ coincide with the identity map (see Remark \ref{rem:y=0})). Moreover, when $(u_k)_{k\in\N}$ is bounded, we set (modulo subsequences)
\begin{equation}
\label{eq:wlimit}
v^{(0)}\eqdef \wlim_{k\to \infty} u_k.
\end{equation}
In what follows we use weak convergence in $H^{1,2}_{\mathrm{loc}}(\R^N)$, which is a weak convergence in a Fr\'echet space. In particular weakly converging sequences will converge in the sense of distributions, strongly in $L^p(\Omega)$ whenever $\Omega\subset \R^N$ is a bounded measurable set and $p\in(2,2^*)$, as well as almost everywhere. 
\begin{theorem}
	\label{thm:MagneticPD} Let  $\Xi\ni 0$ be a discretization of $\R^N$.
	Let $A\in \dot C^1(\R^N,\Lambda_1)$, $N\ge 3$,
	and let	$(u_{k})_{k\in\N}$ be a bounded sequence in $H_A^{1,2}(\R^N)$.
	Then, there exist 
	exist $v^{(n)}  \in
	H^{1,2}_\mathrm{loc}(\R^N)$ (with $v^{(0)}\in H_A^{1,2}(\R^N)$), $Y^{(n)}:=(y_{k}^{(n)})_{k\in\N}\subset \Xi$, $n\in \N$,
	such that, on a renamed subsequence, 
\begin{eqnarray}
\label{mag-shifts}  
&& (g^{(n)}_k)^{-1} u_k \rightharpoonup v^{(n)}\quad \mbox{ in } H^{1,2}_{\mathrm{loc}}(\R^N),\\
&& \label{separates}
|y_{k} ^{(n)}-y_{k} ^{(m)}|\to\infty  \mbox{ for } n \neq m,\\
&& \label{BBasymptotics-mag}
	r_k\eqdef u_{k} - \sum_{n=0}^\infty
		g^{(n)}_kv^{(n)} \to 0\, \mbox{ in }
		L^p(\R^N), \,\forall p\in (2,2^*),
			\end{eqnarray}
and the series $\sum_{n=0}^\infty
		|v^{(n)}(\cdot-y_{k}^{(n)})|$ and the series in (\ref{BBasymptotics-mag}) converge
	 unconditionally (with respect to $n$) and uniformly in $k$, in 
$H^{1,2}(\R^N)$ and in $H^{1,2}_\mathrm{_{loc}}(\R^N)$ respectively.
Moreover, for any $n$,
\begin{equation}\label{eq:gsaturation}
g^{-1}_{y_k}u_k\rightharpoonup 0  \mbox{ in } H^{1,2}_\mathrm{_{loc}}(\R^N)\quad \mbox{ for all } \; (y_k)_{k\in\N} \mbox{ in } \R^N\mbox{ s.t. }|y_k-y^{(n)}_k|\to+\infty.
\end{equation}
\end{theorem}
Note that, since $(y^{(0)}_k)_{k\in\N}=(0)_{k\in\N}$, formula \eqref{separates} implies that each 
sequence $(y^{(n)}_k)_{k\in\N}=(0)_{k\in\N}$ is diverging for every $n\neq 0$, while \eqref{BBasymptotics-mag} gives the profile decomposition
\begin{equation}
\label{eq:PD}
u_{k}=v^{(0)}+\sum_{n=1}^\infty
		g^{(n)}_kv^{(n)} +r_k \quad \mbox{ with } r_k\to 0\, \mbox{ in }
		L^p(\R^N), \,\forall p\in (2,2^*).
\end{equation}

\begin{theorem}\label{thm:energies} Assume conditions of Theorem~\ref{thm:MagneticPD}. A subsequence $(u_{k})_{k\in\N}$ in $H_A^{1,2}(\R^N)$ provided by Theorem~\ref{thm:MagneticPD} satisfies
	\begin{equation}\label{eq:IBL}
	\int_{\R^N}|u_k|^p\dx{x}\longrightarrow \sum_{n=0}^\infty\int_{\R^N}|v^{(n)}|^p\dx{x}\quad \mbox{ as } k\to +\infty\;\forall p\in[2,2^*),
	\end{equation}
		\begin{equation}\label{eq:IBL2}
 \sum_{n=0}^\infty\int_{\R^N}|v^{(n)}|^2\dx{x}\le \liminf_{k\to\infty} \int_{\R^N}|u_k|^2\dx{x},
	\end{equation}
and	furthermore, 
	\begin{equation} 
	\label{norms-mag} 
	\sum_{n =0}^\infty E_{A^{(n)}_\infty}(v ^{(n)})\le \liminf_{k\to \infty} E_A(u_k),
	\end{equation}
where $A^{(0)}_\infty=A$, and for $n\neq 0$, each $A^{(n)}_\infty$ is the magnetic potential at infinity relative to the sequence $Y^{(n)}=(y_k^{(n)})_{k\in\N}$ as in Definition~\ref{def:AY} and $B^{(n)}_\infty=d A^{(n)}_\infty$ is the related magnetic field.
\end{theorem}
%
Before giving the proof of Theorem~\ref{thm:MagneticPD} we like to show that relation \eqref{mag-shifts} is at all possible.

\begin{lemma}\label{lem:onebubble}
Let $A\in \dot C^{1}(\R^N,\Lambda_1)$, $N\ge 3$,
let $(u_k)_{k\in\N}$ be a bounded sequence in $\dot H^{1,2}_A(\R^N)$, and let $Y=(y_k)_{k\in\N}$ be a sequence in $\R^N$.
Then, $(g^{-1}_{y_k}u_k)_{k\in\N}$ has a subsequence weakly convergent in $H^{1,2}_\mathrm{loc}(\R^N)$ to some $v$ in $H^{1,2}_\mathrm{loc}(\R^N)$, such that, $(A_{y_k}(\cdot+y_k))_{k\in\N}$ converges uniformly on bounded sets to some
$A_\infty\in \dot C^{0,1}(\R^N,\Lambda_1)$ such that
 \begin{equation}
 \label{eq:onebubble}
 E_{A_\infty}(v)\le \liminf_{k\to\infty} E_{A}(u_k). 
 \end{equation}
\end{lemma}
\begin{proof} To shorten notation we shall set, for any $k\in\N$,
$$A_k\eqdef A_{y_k}(\cdot+y_k)\quad \mbox{ and }\quad v_k\eqdef g^{-1}_{y_k}u_k.$$
By Lemma~\ref{lem:localbound} the sequence $(v_k)_{k\in\N}$ is bounded in $H^{1,2}_\mathrm{loc}(\R^N)$ and thus it has a subsequence weakly convergent to some
$v\in H^{1,2}_\mathrm{loc}(\R^N)$.
Consider a (renamed) subsequence, given by the Arzel\`a-Ascoli theorem, such that $(A_k)_{k\in\N}$ converges uniformly on bounded sets to $A_\infty\eqdef A_\infty^{(Y)}$ (see Definition \ref{def:AY}).
Set for any $R>0$, $\Omega_R \eqdef B_R(0)$, by applying \eqref{eq:commutation}, we have
\begin{equation*}\label{eq:EAR}
\begin{split}
E_A(u_k)& =E_A(g_{y_k}v_k)\ge 
\int_{\Omega_R}|\nabla_{A_{y_k(\cdot+y_k)}}({g_{y_k}}^{-1}u_k)|^2 \dx{x}\\
& =\int_{\Omega_R}|\nabla_{A_k}(v_k)|^2 \dx{x}\quad \mbox{ for all } R>0.\\
\end{split}
\end{equation*}
Now, let us remark that
\begin{equation*}
\begin{split}
|\nabla_{A_k}v_k|^2& =|\nabla v_k +\i A_\infty v_k-\i (A_\infty-A_k)v_k|^2=
|\nabla_{A_\infty}v_k-\i (A_\infty-A_k)v_k|^2\\
& =
|\nabla_{A_\infty}v_k|^2+\i (A_\infty-A_k)
\left(\overline{v_k}\nabla_{A_\infty}v_k-v_k\overline{\nabla_{A_\infty}v_k} \right)+
|A_\infty-A_k|^2 |v_k|^2,\\
\end{split}
\end{equation*}
and that
$$
\overline{v_k}\nabla_{A_\infty}v_k-v_k\overline{\nabla_{A_\infty}v_k}= 
\overline{v_k}\nabla v_k-v_k \nabla \overline{v_k}+ 2\i A_\infty|v_k|^2.
$$
So, we deduce the following relation
\begin{equation*}
\begin{split}
|\nabla_{A_k}v_k|^2 &=
|\nabla_{A_\infty}v_k|^2-2 A_\infty (A_\infty-A_k) |v_k|^2
+|A_\infty-A_k|^2 |v_k|^2\\
&\quad +
\i (A_\infty-A_k) \left(\overline{v_k}\nabla v_k-v_k \nabla\overline{v_k} \right),\\
\end{split}
\end{equation*}
and, by arguing as in \eqref{products}, we get (by applying Cauchy-Schwarz Inequality)
$$
|\nabla_{A_k}v_k|^2\geq
|\nabla_{A_\infty}v_k|^2-2 |A_\infty| |A_\infty-A_k| |v_k|^2
+|A_\infty-A_k|^2 |v_k|^2
-4 |A_\infty-A_k| |v_k|\,|\nabla v_k|.
$$
Then, setting $\alpha_k\eqdef \|A_k\|_{L^\infty(\Omega_R)}$ and 
$\delta_k\eqdef \|A_\infty-A_k\|_{L^\infty(\Omega_R)}$ we get, by H\"older Inequality, that
\begin{equation*}
\begin{split}
E_A(u_k)& \geq \int_{\Omega_R}|\nabla_{A_\infty}v_k|^2\dx{x} 
-2 \a_k\delta_k \int_{\Omega_R}|v_k|^2\dx{x}+
\delta_k \int_{\Omega_R}|v_k|^2\dx{x}\\
& \quad -
4 \delta_k \|v_k\|_{L^2(\Omega_R)} \|\nabla v_k\|_{L^2(\Omega_R)}\\
& = 
\int_{\Omega_R}|\nabla_{A_\infty}v_k|^2\dx{x} +o(1)
\end{split}
\end{equation*}
where we have taken into account that $\delta_k\to 0$ and the existence of the local bound in $H^{1,2}$ for the sequence $(v_k)_{k\in\N}$ according to Lemma~\ref{lem:localbound}.
So, by weak lower semicontinuity of seminorms, \eqref{eq:onebubble}
follows by taking $R\to \infty$.
\end{proof}
We will use the following statement \cite[Corollary~3.3]{ccbook} (which can also be derived from  Solimini's result \cite[Theorem 2]{Sol}, with the addition of the property on the unconditional convergence of the series as underlined in the Banach space version of the same theorem in \cite{ST1}).
\begin{proposition}
	\label{acc-shifts}
	Let $(u_{k})_{k\in\N}$ be a bounded sequence in $H^{1,2}(\R^N)$.
	Then, for all $n\in\N$, there
	exist $w^{(n)} \in H^{1,2}(\R^N)$, 
	$(y_{k} ^{(n)})_{k\in\N}$ in $\Z^N$,
	such that, on a renamed subsequence,
	\begin{eqnarray}
	\label{w_n-shifts} && u_k(\cdot+y_{k} ^{(n)})\rightharpoonup w^{(n)},
	\\
	&&\label{separates-shifts} |y_{k} ^{(n)}-y_{k} ^{(m)}|\to\infty  \mbox{ for
	} n \neq m,
	\\
	&&\label{norms-shifts} \sum_{n \in \N} \|w ^{(n)}\|_{H^{1,2}}^2 \le
	\limsup_{k\to\infty} \|u_k\|_{H^{1,2}}^2,
	\\
	&&\label{BBasymptotics-shifts} u_{k} - \sum_{n\in\N}
	w^{(n)}(\cdot-y_{k} ^{(n)})  \to 0\, \text{in }
	L^p(\R^N),\,\forall p\in (2,2^*), 
	\end{eqnarray}
and the series in (\ref{BBasymptotics-shifts}) converges
	in $H^{1,2}(\R^N)$ unconditionally (with respect to $n$) and uniformly in $k$.
Moreover, 
\begin{equation}\label{eq:saturation}
u_k(\cdot +y_k) \rightharpoonup 0 \quad \mbox{ for all } (y_k)_{k\in\N} \mbox{ in $\R^N$ s.t. }
|y_k-y^{(n)}_k|\to \infty, \forall n\in\N\,.
\end{equation}
\end{proposition}
\begin{remark}
	\label{rem:acc-shifts}
	One can replace $\Z^N$ in Proposition~\ref{acc-shifts} with any other discretization $\Xi$ of $\R^N$. Indeed, for any sequence $(y_k^{(n)})_{k\in\N}$ in $\Z^N$ there exists a sequence
$(\bar y_k^{(n)})_{k\in\N}$ in $\Xi$, such that $(\bar y_k^{(n)} - y_k^{(n)})_{k\in\N}$ is bounded. Then, passing to a renamed subsequence and by using standard diagonalization, we may assume that 
	$\left(\bar y_k^{(n)}- y_k^{(n)}\right)_{k\in\N}$ converges to some point $z_n\in\R^N$, and  
	all assertions of Proposition~\ref{acc-shifts} will hold with $\bar y_k^{(n)}$ and $\overline{w}^{(n)}\eqdef w^{(n)}(\cdot+z_n)$ replacing $y_k^{(n)}$ and $w^{(n)}$ respectively.
\end{remark}
%
%
%
%
\begin{DIof}{Theorem~\ref{thm:MagneticPD}}
Let $(u_k)_{k\in\N}$ be a bounded sequence in $H^{1,2}_A(\R^N)$, then, by \eqref{eq:diamag}, we can apply Proposition~\ref{acc-shifts}, appended by Remark~\ref{rem:acc-shifts}, to the sequence $(|u_k|)_{k\in\N}$.
Denoting the magnetic shift $g_{y_k^{(n)}}$ determined by $y_k^{(n)}$ as $g_k^{(n)}$ (see Definition \ref{def:g} and \eqref{eq:short}) and the magnetic potential $A_{y_k^{(n)}}(\cdot + y_k^{(n)})$ as $A_k^{(n)}$,
 we have, by Lemma~\ref{lem:onebubble}, that, on a renamed subsequence,  
$\left({(g_k^{(n)}})^{-1}u_k\right)_{k\in\N}$ converges weakly 
in $H^{1,2}_{\mathrm{loc}}(\R^N)$ to some $v^{(n)}\in H_{\mathrm{loc}}^{1,2}(\R^N)$, i.e. 
\begin{equation*}
\label{eq:vn}
v^{(n)}\eqdef\wlim_{k\to\infty} (g_k^{(n)})^{-1}u_k,
\end{equation*}
while $A_k^{(n)}\to A_\infty^{(n)}\in C^{0,1}$ uniformly on bounded sets. 
It follows from the equality
$|{g_k^{(n)}}u_k|=|u_k(\cdot+y_k^{(n)})|$ and from \eqref{w_n-shifts} that $|v^{(n)}|=w^{(n)}$ for all $n$. 
Furthermore, by taking into account \eqref{eq:saturation}, we have, for the same reason as above, that $g_{y_k}u_k\to 0$ whenever $|y_k-y_k^{(n)}|\to\infty$.

It easily follows from convergence properties of the series  in (\ref{BBasymptotics-shifts}) that the series in \eqref{BBasymptotics-mag} converges unconditionally (with respect to $n$) and uniformly in $k$ in  $H^{1,2}_\mathrm{loc}(\R^N)$.

Now, we shall prove \eqref{BBasymptotics-mag} by means of Lemma \ref{lem:vanish}, i.e.\ by proving that the sequence $(r_k)_{k\in\N}$ (therein defined) $\mathcal{G}$-converges to zero.
So, given a sequence $(y_k)_{k\in\N}$ in $\Xi$, we have to prove that, for the corresponding sequence of magnetic shifts $(g_{y_k})_{k\in\N}$, we have
$g^{-1}_{y_k}  r_k \rightharpoonup 0$ in $H^{1,2}_{\mathrm{loc}}(\R^N)$. 
Actually, we have to face two cases.

If (first case), for all $n$, $|y_k-y^{(n)}_k|\to +\infty$ as $k\to +\infty$, we have, by \eqref{eq:gsaturation}, that 
%
not only 
$g_{y_k}u_k\rightharpoonup 0$, 
but also each term in the series of concentrations vanishes as well, i.e.\ $g_{y_k} g^{(n)}_k v^{(n)}\rightharpoonup 0$, so $g_{y_k} r_k \rightharpoonup 0$. 

If otherwise (second case), $(y_k)_{k\in\N}$ in $\R^N$ is such that $(|y_k-y_k^{(n)}|)_{k\in\N}$ is bounded for some $n\in\N$, then passing to a renamed subsequence we may assume that $y_k-y_k^{(n)}\to\overline{y}^{(n)}$ with some $\overline{y}^{(n)}\in\R^N$. Note that, in such a case, $|g_{y_k}(g_k^{(n)})^{-1}u|\to |u(\cdot -\overline y^{(n)})|$ in $H^{1,2}(\R^N)$ and that the same is true to their adjoint operators as well.  
Then, 
denoting by $o^w(1)$ a sequence weakly convergent to zero, we have 
\begin{equation*}
\begin{split}
\left|
g_{y_k}\left(u_k-(g_k^{(n)})^{-1}v^{(n)}\right)
\right|& =
\left|
g_{y_k}(g_k^{(n)})^{-1}({g_k^{(n)}}u_k-v^{(n)})
\right|\\
\\
& =
\left|
\left(g_k^{(n)}u_k-v^{(n)}\right)(\cdot-\overline{y}^{(n)})
\right|+o^w(1)\rightharpoonup 0 \;,
\end{split}
\end{equation*}
from which we conclude that $g_{y_k} r_k \rightharpoonup 0$.
Then, 
by Lemma~\ref{lem:vanish}, we deduce that
$r_k\to 0$ in $L^p(\R^N)$, for all $p\in(2,2^*)$ getting \eqref{BBasymptotics-mag}.
\end{DIof}
\vskip2mm
\begin{DIof}{of Theorem~\ref{thm:energies}}
The first assertion is an iterated version of Brezis-Lieb Lemma for the sequence $(|u_k|)_{k\in\N}$ and its profile decomposition in $H^{1,2}(\R^N)$ (see e.g. \cite{CwiTi} : note that $|u_k|(\cdot-y_k^{(n)})=|g_k^{(n)}u_k|\rightharpoonup |v^{(n)}|$ in $H^{1,2}(\R^N)$ since for sequences bounded in $H^{1,2}(\R^N)$ weak and $a.e.$ convergence coincide). 

Let us prove the third assertion.
Following the proof of Lemma~\ref{lem:onebubble} to $v=v^{(n)}$, $y_k=y_k^{(n)}$, we have
\begin{equation}\label{eq:LowerBound}
\int_{B_R(y_k^{(n)})}|\nabla_Au_k|^2\dx{x}\ge \int_{B_R(0)}|\nabla_{A^{(n)}_\infty}v^{(n)}|^2\dx{x} +o_{k\to\infty}(1).
\end{equation}
On the other hand, due to \eqref{separates-shifts}, we can assume that, for any $R>0$, for $k$ large enough,
$B_R(y^{(n)}_k)\cap B_R(y^{(m)}_k)=\emptyset$ for all $m\neq n$, so for any $M\in\N$, $M\neq 0$,
$E_A(u_k)\geq \sum_{n=0}^M\int_{B_R(y^{(n)}_k)}|\nabla_A u_k|^2 \dx{x}$. So, by \eqref{eq:LowerBound}, we get
\begin{align*}
\liminf_{k\to\infty} E_A(u_k) \ge
\sum_{n=0}^{M}\int_{B_R(0)}|\nabla_{A^{(n)}_\infty}v^{(n)}|^2\dx{x}.
\end{align*}
Then, \eqref{norms-mag} follows since $M$ and $R$ are arbitrary.
\end{DIof}

\section{Applications}
Given a bounded electric potential $V$, described as a positive measurable function on $\R^N$, we introduce the following functional defined by setting for all $u\in H^{1,2}_A(\R^N)$
\begin{equation}\label{eq:JAV}
J_{A,V}(u)=\int_{\R^N}\left(|(\nabla_A u(x)|^2+V(x)|u(x)|^2\right)\dx{x}.
\end{equation}
We shall emphasize the case in which the electric potential is constant, i.e.\ when $V\equiv \lambda\in \R$, by using the notation $J_{A,\lambda}$ instead of 
$J_{A,V}$.

\subsection{A ground state problem}
Set $\lambda_0\eqdef \inf_{\|u\|_2=1}E_A(u)$, we shall introduce, for $\lambda>-\lambda_0$,
the following minimization problems:
\begin{equation}\label{eq:inf0}
C(p,N,A,\lambda)\eqdef \inf_{\stackrel{u\in H^{1,2}_A(\R^N)}{\|u\|_p=1}}J_{A,\lambda}(u),\; \quad p\in(2,2^*).
\end{equation}
We will show that, for any $p\in(2,2^*)$, the problem \eqref{eq:inf0} attains no minimum if the magnetic field $B=dA$ vanishes at infinity.

It is worth to remark that several known existence results (see \cite{ArioliSzulkin}, \cite{Kurata}, \cite{SchinTinMag}) prove existence of minimum when the electric potential increases at infinity and offsets the possible decrease of energy (due to the magnetic field) at infinity, penalizing (in minimizing sequences) any possible translation towards infinity.
Of course, this setting excludes the constant potential $V\equiv \lambda$.
Below, we show in an elementary statement, that if the magnetic field at infinity is zero and the electric potential $V\equiv \lambda$ is constant (so there is no penalty mechanism) the ground state does not exist. Note that the statement is restricted only to positive values of $\lambda$, leaving it as an elementary exercise for the reader to prove that, when $V\equiv \lambda\le 0$, the infimum in \eqref{eq:inf0} is zero.

\begin{proposition}\label{prop:inf}
Let $N\ge 3$ and let $B=d A$ be a bounded magnetic field with $A\in C^{1,\a}$ for some $\a\in(0,1)$. Assume moreover that $\mathrm{supp} B=\R^N$ and that $\lim_{|x|\to\infty}B(x)=0$. Then, for all $\lambda>0$, Problem \eqref{eq:inf0} has no solution, in particular the infimum is positive but it is not attained.
\end{proposition}
\begin{proof} 
Nonnegativity of the infimum easily follows by the assumption $\lambda>0>-\lambda_0$, while positivity follows from diamagnetic inequality \eqref{eq:diamag} and from the Sobolev embedding applied to $|u|$.

Assuming, by contradiction, that the infimum in \eqref{eq:inf0}
is attained at some function $u$, we get $C(N,p,A,\lambda)>C(N,p,0,\lambda)$. 
Indeed, by taking into account that the diamagnetic inequality \eqref{eq:diamag} is strict on $\mathrm{supp} B\cap\mathrm{supp}u$, and by using \eqref{eq:Aineq}, we have
\begin{equation}
\label{eq:CA>C0}
C(N,p,A,\lambda)=J_{A,\lambda}(u)>J_{0,\lambda}(|u|)=J_{0,\lambda}(u)\ge C(N,p,0,\lambda). 
\end{equation}
On the other hand, since the constant $C(N,p,0,\lambda)$ (see \eqref{eq:inf0} with $A\equiv 0$) is attained by the well-known minimizer $w_0$, (which decays exponentially at infinity), and since, by \eqref{eq:Aineq}, we have that the magnetic potential $A^{(Y)}_\infty$ (see Definition \ref{def:AY}) associated to any diverging sequence $Y\eqdef (y_k)_{k\in\N}\subset\R^N$ is zero, we have 
\[
C(N,p,A,\lambda)\le J_{A,\lambda}(w_0(\cdot-y_k))\to J_{0,\lambda}(w_0)=C(N,p,0,\lambda),
\] 
getting in contradiction to \eqref{eq:CA>C0}.
\end{proof}

%
%

\subsection{A minimax problem}
Let $N\geq 3$, $\lambda>0$ and $p\in (2,2^\ast)$. 
We consider on $H^{1,2}_A(\R^N)$ the following Hamiltonian functional, (see \eqref{eq:JAV}):
	\begin{equation}\label{eq:HamiltonianT}
	\begin{split}
	I_{A,\lambda}(u)& =\frac12 J_{A,\lambda}(u)-\frac{1}{p}\int_{\R^N}|u|^p\dx{x}\\
	& =\frac{1}{2}\left(E_A(u)+\lambda \|u\|_2^2\right) -\frac{1}{p}\|u\|_p^p,
	\end{split}
	\end{equation}
Since in this section we shall deal with the case in which the magnetic potential $A$ vanishes at infinity, we shall set
\begin{equation}\label{eq:Hamiltonian-inftyT}
\begin{split}
I_{\infty}(u)& \eqdef I_{0,\lambda}(u) =\frac12\int_{\R^N}\left(|\nabla u(x)|^2+\lambda |u(x)|^2\right)\dx{x}-\frac{1}{p}\int_{\R^N}|u|^p\dx{x}\\
&= \frac{1}{2}\left(E_0(u)+\lambda \|u\|_2^2\right) -\frac{1}{p}\|u\|_p^p.
\end{split}
\end{equation}
This functional is defined on defined on $H^{1,2}(\R^N)$.

Note that critical points of the functional $I_{A,\lambda}$ and $I_\infty$ are weak solutions to the problems
\begin{equation}
\label{eq:P}
\tag{$P$}
\begin{cases}
-\nabla^2_A u+\lambda u = |u|^{p-2}u\\
u\in H^{1,2}_A(\R^N),
\end{cases}
\end{equation}
and
\begin{equation}
\label{eq:Pinfty}
\tag{$P_\infty$}
\begin{cases}
-\Delta u+\lambda u = |u|^{p-2}u\\
u\in H^{1,2}(\R^N),
\end{cases}
\end{equation}
respectively.
In particular, since every weak nontrivial solution to \eqref{eq:Pinfty} belongs to the Nehari manifold
\begin{equation*}
\label{eq:Nehari*}
\mathcal{N}_\infty\eqdef \{u\in H^{1,2}(\R^N)\setminus \{0\}\;|\; J_{0,\lambda}(u)=\|\nabla u\|_2^2+\lambda \|u\|_2^2=\|u\|_p^p\},
\end{equation*}
we call ground state level of the functional $I_\infty$ the quantity
$$c_\infty\eqdef \inf_{\mathcal{N}_\infty} I_\infty\,.$$
It is well-known that $c_\infty$ is achieved by a unique (up to translations and change of sign) function $w_\infty$ (to which one refers as to the ground state solution to \eqref{eq:Pinfty}) which is
smooth, positive and radial decreasing with exponential decay at infinity. In particular, we have
\begin{equation}
\label{eq:Nehari}
J_{0,\lambda}(w_\infty)=\|\nabla w_\infty\|_2^2+\lambda \|w_\infty\|_2^2=\|w_\infty\|_p^p,
\end{equation}
and, as a consequence,
\begin{equation*}
\label{eq:c_inftyT}
c_\infty=\frac{p-2}{2p} \|w_\infty\|_p^p, 
\end{equation*}
or, conversely,
\begin{equation}
\label{eq:norma_inftyT}
\|w_\infty\|_p^p= \frac{2p}{p-2} c_\infty.
\end{equation}
It is possible to see $c_\infty$ as a $\min-\max$ level. Indeed, it is standard to prove that for any $u\in H^{1,2}(\R^N)\setminus\{0\}$ there exists a unique scaling factor $\overline{t}>0$ such that $\hat{u}\eqdef \overline{t} u\in \mathcal{N}_\infty$ and 
$$I_\infty(\hat{u})=\max_{t\in[0,\infty)}I_\infty(t u).$$
So,
\begin{equation*}\label{eq:mmclassic}
c_\infty=I_\infty(w_\infty)=\inf_{u\in H^{1,2}(\R^N)\setminus\{0\}}I_\infty(\hat{u})=
\inf_{u\in H^{1,2}(\R^N)\setminus\{0\}} \max_{t\in[0,\infty)}I_\infty(t u).
\end{equation*}
Now, by taking the set $\Phi\subset C([0,\infty),H^{1,2}(\R^N))$ consists of all paths satisfying 
$\sigma(0)=0$ and $\lim_{t\to\infty} |\sigma(t)|=+\infty$ (and, as a consequence,  $\lim_{t\to\infty} I_\infty(\sigma(t))=-\infty$) we have the following further characterization of $c_\infty$
%
%
\begin{equation*}\label{eq:mm-inftyT}
c_\infty\eqdef\min_{\sigma\in\Phi}\max_{t\in[0,\infty)}I_\infty(\sigma(t)).
\end{equation*}

Finally, it is worth to remark that the ground state $w_\infty$ equals, up to a scalar multiple, 
the (radial) minimizer of $C(p,N,0,\lambda)$ defined in \eqref{eq:inf0}, so
$w_0=\frac{w_\infty}{\|w_\infty\|_p}$, i.e.\
\begin{equation*}
\label{eq:J=k}
J_{0,\lambda}\left(\frac{w_\infty}{\|w_\infty\|_p^p}\right)=C(p,N,0,\lambda).
\end{equation*}
\vskip2mm
Let $A\in \dot C^1(\R^N,\Lambda_1)\setminus\{0\}$ and assume the following conditions:
\begin{itemize}
	\item[$\mathbf{(A)}$] $B_\infty^{(Y)}=0$ $a.e.$ for any diverging sequence $Y\eqdef (y_k)_{k\in\N}\subset\R^N$,
	(equivalently, $\lim_{|y_k|\to\infty} A_{y_k}(\cdot+y_k)=0$);\\
	\item[$\mathbf{(B)}$]  $\|B\|_\infty\le
	\left(\dfrac{(2^\frac{p-2}{p}-1)\|w_\infty\|_p^p}{\int_{\R^N}|x|^2 w_\infty^2\dx{x}}\right)^\frac12\;.$
\end{itemize}

Let us define a map $\eta:=(\eta_1,\dots,\eta_{N+1}): H_A^{1,2}(\R^N)\to \R^{N+1}$ by setting:
\begin{equation}
\eta_i(u)=
\begin{cases}
\int_{\R^N}\frac{x_i}{1+|x|}|u|^p\dx{x},& i=1,\dots,N,
\\
\int_{\R^N}|u|^p\dx{x}, & i=N+1.
\end{cases}
\end{equation}
Note that, since $w_\infty$ is radially symmetric,
\begin{equation}
\label{eq:eta0}
\eta_0 \eqdef (0_N,\frac{2p}{p-2}c_\infty)=(0_N,\|\omega_\infty\|_p^p)\in \R^{N+1}.
\end{equation}

Let us fix $T\ge 2$ sufficiently large so that $I_\infty(Tw_\infty)<0$, and
assume that $R>0$ is large enough  (in what follows the value of $R$ will be subjected to a finite number of restrictions). Let $B_R$ denote the closed ball in $\R^N$ centered at the origin with radius
$R>0$, and let 
\begin{equation}\label{eq:pathsT}
\Gamma=\Gamma_{R,T}\eqdef \{\gamma\in C(B_R\times[0,T],\R^{N+1})\; |\; 
\gamma=\gamma_0, \mbox{ on }\partial B_R\times[0,T]\},
\end{equation}
where 
\begin{equation}
\label{eq:gamma0}
\gamma_0(y,t)=t \,e^{\i\varphi_y}w_\infty{(\cdot-y)}=t g_y w_\infty,\quad \forall (y,\,t)\in B_R\times [0,\,T],
\end{equation}
and $\varphi_y$ is given by Lemma \ref{lem:phase2} (see \eqref{eq:g}).
Taking $\eta_0$ as in \eqref{eq:eta0}, we have that, by construction,
\begin{equation}\label{eq:degreeT}
\mathrm{deg} (\eta\circ\gamma,B_R\times[0,T],\eta_0)=
\mathrm{deg} (\eta\circ\gamma_0,B_R\times[0,T],\eta_0)\neq 0\quad \forall \gamma\in\Gamma.
\end{equation}
(Note that the degree of $\eta\circ\gamma_0$ to if well-defined. Indeed, by the choice of $T$, we have
$I_\infty(Tw_\infty)<0$ and, since $w_\infty$ is radially symmetric, we have $\eta\circ\gamma_0(y,t)\neq 0$ whenever $y\neq 0$.) 
Then, provided $R$ and $T$ are sufficiently large, the following number is well defined
\begin{equation}\label{eq:mm}
c_{R,T}\eqdef\inf_{\gamma\in\Gamma}\max_{y\in B_R, t\in[0,T]}I_{A,\lambda}(\gamma(y,t)).
\end{equation}

\begin{lemma}\label{lem:>cinfT}
If $A\in C^1(\R^N,\Lambda_1)\setminus\{0\}$, then
\begin{equation}\label{eq:>cinfT}
c_{R,T}> c_\infty,
\end{equation}
provided $R$ and $T$ are sufficiently large real numbers.
\end{lemma}
\begin{proof} 
	By \eqref{eq:degreeT}, and by taking into account that the last component of $\eta_0$ is $\|w_\infty\|^p_p=\frac{2pc_\infty}{p-2}$, we have
	\begin{equation*}
	\label{eq:cinequalities}
	\begin{split}
	c_{R,T} & \ge\inf_{\eta(u)= \eta_0}
	\left(\frac{1}{2}J_{A,\lambda}(u)-\frac{1}{p}\|u\|^p_p\right)
		\\
	& = \inf_{\eta(u)= \eta_0}\frac12 J_{A,\lambda}(u)- \frac{2}{p-2}c_\infty
\\
& \ge \inf_{\|u\|_p=\|w_\infty\|_p}\frac12 J_{A,\lambda}(u)- \frac{2}{p-2}c_\infty
\\
& \ge \inf_{\|u\|_p=\|w_\infty\|_p}\frac12 J_{0,\lambda}(|u|)- \frac{2}{p-2}c_\infty=c_\infty,
	\end{split}
	\end{equation*}
	where, in the last equality, we have used \eqref{eq:Nehari} and \eqref{eq:norma_inftyT}.
	Then, the thesis follows since the diamagnetic inequality \eqref{eq:diamag} is strict on any minimizer
	$\pm w_\infty (\cdot-y)$, $y\in\R^N$, of $c_\infty$ (since $\supp B\cap \supp w_\infty=\R^N$).
\end{proof}
\begin{lemma}\label{lem:<2cT}
If $\mathbf{(B)}$ holds true, then
\begin{equation}\label{eq:<2cT}
	c_{R,T}< 2c_\infty,
\end{equation}
provided $R$ and $T$ are sufficiently large real numbers.
\end{lemma}
\begin{proof}
Since $\gamma_0$ (see \eqref{eq:gamma0}) trivially belongs to the set $\Gamma$ (see \eqref{eq:pathsT}), 
set
\[
\sigma\eqdef \frac{\|B\|_\infty^2\int|x|^2w_\infty^2\dx{x}}{\|w_\infty\|_p^p},
\]
we get, by \eqref{eq:Aineq}, that
\begin{equation*}
\begin{split}
c_{R,T}&\le\sup_{y\in B_R, t\in[0,T]} \frac12 t^2 J_{A,\lambda}(g_yw_\infty)-\frac1p t^p\|w_\infty\|_p^p\\
&= \sup_{y\in B_R, t\in[0,T]} \frac12 t^2 J_{A_y(\cdot+y),\lambda}(w_\infty)-\frac1p t^p\|w_\infty\|_p^p\\
& \le \sup_{t\in[0,T]} \frac12 t^2 J_{0,\lambda}(w_\infty)-\frac1p t^p\|w_\infty\|_p^p +
t^2\|B\|_\infty^2\int_{\R^N}|x|^2w_\infty^2\dx{x}\\
& =\sup_{t\in[0,T]} \frac12 t^2 \|w_\infty\|_p^p -\frac1p t^p\|w_\infty\|_p^p +t^2\sigma\|w_\infty\|_p^p\\
& =\|w_\infty\|_p^p\sup_{t\in[0,T]} \left(\frac12 t^2 (1+\sigma) -\frac1p t^p\right)\\
& = \|w_\infty\|_p^p (1+\sigma)^{\frac{p}{p-2}}\left(\frac12-\frac1p\right)
=c_\infty (1+\sigma)^{\frac{p}{p-2}},
\end{split}
\end{equation*}
where, in the last equality, we have used \eqref{eq:norma_inftyT}.
Then, the thesis follows since assumption $(\mathbf{B})$ just implies $(1+\sigma)^\frac{p}{p-2}<2$.
\end{proof}
\begin{remark}Note that without assumption $(\mathbf{B})$ one has $c_{R,T}\le 2c_\infty+o_{R\to\infty}(1)$. This can be shown by considering the inequality
	\[
	c_{R,T}\le\sup_{y\in B_R, t\in[0,T]} I_{A,\lambda}(\gamma(y,t))
	\]
with $\gamma\in \Gamma$ defined by setting, for all $(y,\,t)\in B_R\times [0,\,T]$,
\[
\gamma(y,t)=t\cos\left(\frac{\pi |y|}{2R}\right)g_{-\frac{R}{|y|}y}w_\infty
+t\sin\left(\frac{\pi |y|}{2R}\right)g_{\frac{R}{|y|}y}w_\infty.
\]	
\end{remark}

\begin{theorem}\label{thm:minimaxT}
	Assume $(\mathbf{A})$ and $(\mathbf{B})$. Then, provided $R>0$ and $T\ge 2$ are sufficiently large, the number $c_{R,T}$ defined by \eqref{eq:mm} is a critical level for the functional $I_{A,\lambda}$ for any $\lambda>0$.
	In particular Problem \eqref{eq:P} admits a solution at the energy level
	$c_{R,T}$.
\end{theorem}
Before giving the proof of the theorem we shall discuss about profile decompositions of a Palais Smale sequence (P.S. for short) for the functional $I_{A,\lambda}$ at the level $c_{R,T}$ given by \eqref{eq:mm}.

Let $(u_k)_{k\in\N}$ be a P.S. sequence for $I_{A,\lambda}$ in $H_A^{1,2}(\R^N)$ 
at the level $c_{R,T}$, i.e.\ such that $I_{A,\lambda}(u_k)\to c$ and $I^\prime_{A,\lambda}(u_k)\to 0$ as $k\to \infty$.
By applying the nowadays standard argument from \cite{AmbroRabin} we deduce that the sequence $(u_k)_{k\in\N}$ is bounded in $H_A^{1,2}(\R^N)$. So we can use a profile decomposition (see \eqref{eq:PD}) of $(u_k)_{k\in\N}$ given by Theorem~\ref{thm:MagneticPD}.
Since each sequence $(y^{(n)}_k)_{k\in\N}$ is diverging for $n\neq 0$ we deduce, by assumption $(\mathbf{A})$, that each magnetic potential $A^{(n)}_\infty$ associated to each sequence $(y^{(n)}_k)_{k\in\N}$ is $0$ if $n\neq 0$ (while $A^{(0)}_\infty=A$). So, \eqref{norms-mag} gives
$$E_A(v^{(0)})+\sum_{n=1}^\infty E_0(v^{(n)})\leq \liminf_{k\to\infty} E_A(u_k),$$
where $v^{(0)}$ and $v^{(n)}$ are defined by \eqref{eq:wlimit} and \eqref{mag-shifts} respectively.
Then, by taking into account \eqref{eq:HamiltonianT} and \eqref{eq:Hamiltonian-inftyT}, we get, by using \eqref{eq:IBL}, that
\begin{equation}\label{eq:mc0}
\begin{split}
I_{A,\lambda}(v^{(0)})+\sum_{n=1}^\infty I_\infty(v^{(n)})& \le 
\liminf_{k\to\infty} \frac12 \left(E_A(u_k)+\lambda \|u_k\|_2^2\right)-\frac1p \|u_k\|_p^p \\
& =\lim_{k\to\infty}I_{A,\lambda}(u_k)=c_{R,T}.
\end{split}
\end{equation}
Note that, since $I^\prime_{A,\lambda}$ is a weak-to-weak continuous functional and $(u_k)_{k\in\N}$ is a Palais Smale sequence, we deduce that its weak limit $v^{(0)}$ is a solution to \eqref{eq:P}. Analogously, due to assumption $(\mathbf{A})$, the other profiles $v^{(n)}$, for $n\neq 0$, are solutions to \eqref{eq:Pinfty}. Therefore, either $v^{(n)}$ is trivial (actually it is not a profile) or $I_\infty(v^{(n)})\geq c_\infty$.
As a consequence, \eqref{eq:mc0} implies that the number $m$ of (nontrivial) profiles is finite and that
%
%
\begin{equation}\label{eq:mc}
mc_\infty\leq I_{A,\lambda}(v^{(0)})+mc_\infty
= I_{A,\lambda}(v^{(0)})+\sum_{n=1}^m I_\infty(v^{(n)}) = c_{R,T},
\end{equation}
where last equalities hold (instead of inequalities) since the profile decomposition of $(u_k)_{k\in\N}$ is finite and each $v^{(n)}$ is a solution to \eqref{eq:Pinfty}. 
\begin{claim}\label{cla:ineq3}
 Assume $(\mathbf{A})$ and $(\mathbf{B})$. Then,
	\begin{equation}\label{eq:ineq3}
	c_0\eqdef I_{A,\lambda}(v^{(0)})\ge c_\infty.
	\end{equation}
\end{claim}
\begin{proof}
	Assume first that $v^{(0)}=0$.
	By combining \eqref{eq:mc}, \eqref{eq:<2cT} and \eqref{eq:>cinfT} we have necessarily $1<m<2$, a contradiction that shows that $v^{(0)}\neq 0$.

Recalling that $w_0=\frac{w_\infty}{\|w_\infty\|_p}$ is a minimizer for $J_{0,\lambda}$ over all vectors with unit $L^p$-norm,
we have
\[
J_{0,\lambda}\left(\frac{w_\infty}{\|w_\infty\|_p}\right)\le J_{0,\lambda}\left(\frac{|v^{(0)}|}{\|v^{(0)}\|_p}\right),
\] 
and, by applying the diamagnetic inequality, we get  
\begin{equation}\label{eq:compare}
J_{0,\lambda}\left(\frac{w_\infty}{\|w_\infty\|_p}\right)\le J_{A,\lambda}\left(\frac{v^{(0)}}{\|v^{(0)}\|_p}\right).
\end{equation}
By taking into account that $v^{(0)}$ and $w_\infty$ are solutions to \eqref{eq:P} and \eqref{eq:Pinfty} respectively, we have $J_{0,\lambda}(v^{(0)})=\|v^{(0)}\|_p^p$ and that (see \eqref{eq:Nehari}) $J_{A,\lambda}(w_\infty)=\|w_\infty\|_p^p$.
Then, from \eqref{eq:compare},
follows that
\begin{equation*}
\frac{2p}{p-2}c_\infty=J_{0,\lambda}(w_\infty)\le J_{A,\lambda}(v^{(0)})=\frac{2p}{p-2}I_{A,\lambda}(v^{(0)}),
\end{equation*}
which immediately gives \eqref{eq:ineq3}.
\end{proof}
\begin{DIof}{of Theorem~\ref{thm:minimaxT}}
Let $R$ and $T$ be large enough to apply Lemma \ref{lem:>cinfT} and Lemma \ref{lem:<2cT} and get that $c\in (c_\infty,2c_\infty)$. Since the diamagnetic inequality is strict on $\supp B\cap \supp v^{(0)}$, we deduce, since $v^{(0)}\neq 0$, that
$$c_0= I_{A,\lambda}(v^{(0)})>I_{0,\lambda}(v^{(0)})\equiv I_\infty (v^{(0)})\geq c_\infty.$$
So, \eqref{eq:mc0} implies that 
$(m+1)c_\infty \le m c_\infty + c_0 = c_{R,T} < 2c_\infty$, i.e.\ that $m=0$. Then, the thesis follows since \eqref{eq:mc} gives, for $m=0$, $c_0=I_{A,\lambda}(v^{(0)})=c_{R,T}$.
\end{DIof}

\vspace{0.5cm}
In conclusion we give a straightforward generalization of Theorem~\ref{thm:minimaxT} for a magnetic
Schr\"odinger equation that also includes a nonconstant electric potential. The proof, which follows step by step the one of Theorem~\ref{thm:minimaxT}, is left to the reader. 
\begin{theorem}\label{thm:minimaxV}
	Let $A\in \dot C^1(\R^N,\Lambda_1)\setminus\{0\}$, $V\in L^\infty(\R^N)$, and assume condition $(\mathbf{A})$ as well as the following conditions:
	\begin{itemize}
		\item[$\mathbf{(B')}$]  
		$\|B\|^2_\infty\int_{\R^N}|x|^2w_\infty^2\dx{x}+\int_{\R^N}(V(x)-\lambda)w_\infty^2\dx{x}\le (2^\frac{p-2}{p}-1)\frac{2p}{p-2}c_\infty$,\\
		\item[$\mathbf{(V)}$] $\lambda\eqdef\lim_{|y|\to\infty} V(y)\le V(x)$, for a.e. $x\in\R^N$.
	\end{itemize}
	Then, there exists a solution $u\in H_A^{1,2}(\R^N)$ to the equation
	\begin{equation}
	-\nabla_A^2 u+Vu=|u|^{p-2}u,
	\end{equation}
	satisfying $I_{A,V}(u)=c'_{R,T}>0$, where 
	\[
	I_{A,V}(u)\eqdef \frac12 J_{A,V}(u)-\frac1p\int_{\R^N}|u|^p\dx{x}	
	\]
	and $c'_{R,T}$ is given by the relation \eqref{eq:mm} (with the functional $I_{A,\lambda}$ replaced by $I_{A,V}$) with $R>0$ and $T\ge 2$ sufficiently large.
\end{theorem}

\section*{Appendix: the group of magnetic shifts}
Magnetic shifts of a lattice-periodic magnetic field do not generally form a group. Let us look at that in more detail.
Note first that, for every $y_1,y_2\in \Z^N$, by using \eqref{eq:etaprod}, we have
\begin{align}\nonumber
g_{y_2}g_{y_1}u=
e^{\i\varphi_{y_2}}e^{\i\varphi_{y_1}(\cdot-y_2)}u(\cdot-y_1-y_2)=
\\
\label{eq:MGclosed}
e^{-\i\gamma(y_1,y_2)}e^{\i\varphi_{y_1+y_2}}u(\cdot-y_1-y_2)
=e^{-\i\gamma(y_1,y_2)}g_{y_1+y_2}.
\end{align}
In other words, a product of magnetic shifts is a magnetic shift up to a constant scalar multiple of magnitude 1. This proves that a larger set, namely
\begin{equation*}
\label{eq:magshift}
\mathcal G_{A}\eqdef\{g_{y,\theta}: u\mapsto e^{\i\theta}e^{\i\varphi_y}u(\cdot-y), u\in H_A^{1,2}(\R^N)\}_{y\in \Z^N, \theta\in\R},
\end{equation*} 
is closed under multiplication law. We see below that $\mathcal G_A$ is a  group. Without loss of generality, since every function $\varphi_y$, $y\in\Z^N$ which satisfies \eqref{eq:phiy}, is defined up to a constant, we may fix its value at a given point, which we choose as follows:  
\begin{equation}\label{eq:givenpoint}
\varphi_y(y/2)=0, \quad y\in\Z^N.
\end{equation}
Then $\varphi_0(0)=0$ and since $\nabla\varphi_0=A(\cdot-0)-A(\cdot)=0$, we have
\begin{equation}\label{zero}
\varphi_0(x)=0, \quad x\in\R^N. 
\end{equation}
From \eqref{eq:etaprod} with $y_2=0$ it follows immediately that for every $y\in\Z^N$ 
\begin{equation*}\label{eq:gammavalues0}
\gamma(y,0)=\gamma(0,y)=0.\\
\end{equation*}
Moreover, by evaluating \eqref{eq:etaprod} with $y_1=y$, $y_2=-y$ at $x=-y/2$, and by using 
\eqref{eq:givenpoint},
we have 
\begin{equation}\label{eq:gammavalues}
\gamma(y,-y)=\varphi_{y}(y/2)+\varphi_{-y}(-y/2) +\gamma(y,-y)=\varphi_0(-y/2)=0
\end{equation}
for every $y\in\Z^N$.

\begin{lemma} For every $y\in\Z^N$, $\theta\in \R$:
	\begin{equation} \label{inv}
	g_{y,\theta}^{-1}=g_{-y,-\theta}\,.
	\end{equation} 
\end{lemma}
\begin{proof} 
Note first that it is sufficient to prove \eqref{inv} for $\theta=0$. Fixed $y\in\Z^N$, from 
\eqref{zero} and by applying \eqref{eq:etaprod} with $y_1=y$, $y_2=-y$, we get by \eqref{eq:gammavalues}
	\begin{equation*}
	\label{inverse-eta}\varphi_{-y}=-\varphi_{y}(\cdot+y)\,.
	\end{equation*}
	Then,
	solving the equation $g_{y,0} u=v$, one has
	\[
	u=e^{-\i\varphi_y(\cdot+y)}v(\cdot+y)=e^{\i\varphi_{-y}}v(\cdot+y)=g_{-y}v.
	\]
\end{proof}	
Note that by \eqref{eq:shifts} the set $\mathcal G$, and thus, the set $\mathcal G_A$, consists of isometries on $H^{1,2}_A(\R^N)$. 

Finally, we can see that $\mathcal G_A$ is a multiplicative group because it is closed with respect to multiplication (by \eqref{eq:MGclosed}) which is trivially associative; it has $g_{0,0}$ as neutral element (see \eqref{zero}); and any of its elements $g_{y,\theta}$ has its own inverse $g_{-y,-\theta}$ (see \eqref{inv}).

\subsubsection*{Acknowledgment} The authors express their warm gratitude to Sergio Solimini for valuable discussions concerning 
Theorem~\ref{thm:minimaxT}.

\bibliographystyle{amsplain}

\end{document}